\documentclass[12pt,leqno,fleqn]{amsart}  
\usepackage{amsmath,amstext,amsthm,amssymb,amsxtra,tikz,comment,bbm}
\usepackage{geometry} 
\usepackage{euler}   
\newcounter{para}
\newcommand{\Para}{\medskip\textbf{\thepara.} \stepcounter{para}} 
\usepackage{mathtools}
    \mathtoolsset{showonlyrefs,showmanualtags} 
\usepackage{hyperref} 
\hypersetup{colorlinks=true, linkcolor=blue, citecolor=magenta, filecolor=magenta, urlcolor=cyan}
\usepackage[msc-links]{amsrefs}  
\allowdisplaybreaks

\theoremstyle{plain} 
\newtheorem{lemma}[equation]{Lemma} 
 
\newtheorem{theorem}[equation]{Theorem}

\theoremstyle{definition}
 
\theoremstyle{remark}
\newtheorem{remark}[equation]{Remark}

\numberwithin{equation}{section}

\title[Square Functions] {Lower Estimate on Square Function of an Indicator Set}
\author[C. Giannitsi]{Christina Giannitsi}   
\address{ School of Mathematics, Georgia Institute of Technology, Atlanta GA 30332, USA}
\email {cgiannitsi3@math.gatech.edu}

\author[M.T. Lacey]{Michael T. Lacey}   
\address{School of Mathematics, Georgia Institute of Technology, Atlanta GA 30332, USA}
\email {lacey@math.gatech.edu}
\thanks{The author is a 2020 Simons Fellow. Research of both authors is supported in part by grant  from the US National Science Foundation, DMS-1949206 and DMS-2247254.}

\begin{document}
\begin{abstract} 
Let $Sf$ be a discrete martingale square function. 
Then, for any set $V$ of positive probability, we have $\mathbb{E} S(\mathbf{1}_V)^2 \geq \eta \mathbb{P}(V)$ 
for an absolute constant $\eta >0$.  We extend this to wavelet square functions, and discuss some related open questions. 
\end{abstract} 
\maketitle 
 \tableofcontents  

\section{Introduction} 
\label{sec:introduction}

The distribution of the Hilbert transform of the indicator of a set $V\subset \mathbb{R}$ 
is  only a function of the measure  $ \lvert V\rvert $ of $V$.  
This is a long standing fact, and we will recall it and recent developments below. 

There is no corresponding known fact for discrete variants of the Hilbert transform.  We establish such results for certain square functions:  An $L^p$ norm of a (discrete) square function of $\mathbf{1}_V$  is necessarily large on $V $ itself.   

The most natural setting for such an inequality  is that of  discrete martingales. We do this in the next section.  We then state and prove a result for square functions associated to wavelets. 
The concluding section includes some history, and points to some closely related questions.

\section{Martingales}
\label{sec:martingales}
The natural setting for this paper is that of discrete martingales.  
Let $(\Omega, \mathcal F , \mathbb{P})$ be a probability space, 
and $\mathcal F_n$ an increasing sequence of $\sigma$-algebras, generated by 
atoms.  Thus,   elements of $\mathcal F_n$ are unions 
of disjoint sets  $A \in \mathcal F_{n+1}$, with $\mathbb{P}(A) >0$. 

Given a function $f \in L^2 (\Omega) $, the associated martingale, and martingale difference sequence are defined by 
\begin{equation}
f_n = \mathbb{E} ( f \mid \mathcal F_n), 
\qquad 
d_n = f_n - f _{n-1}.  
\end{equation}
The square function of $f$ is $S f ^2 = \sum_n d_n^2 $. 

\begin{theorem} \label{t:mart} There is an $\eta >0$ so that for all martingales 
as above, and $V\subset \Omega$ of positive probability, there holds 
\begin{equation}
 \label{e:Ebig} 
 \mathbb{E} \mathbf{1}_V  (S \mathbf{1}_V)^2 \geq \eta \mathbb{P}(V). 
 \end{equation}
  
\end{theorem}

The approach is to assume that \eqref{e:Ebig} fails for some choice of probability space $\Omega $,  filtration $\mathcal F_n$, corresponding martingale $\{d_n\}$ 
and set $V$.  From this, we establish that $\eta>0$ admits an absolute lower bound. 
Take $\{X_n\}$ to be a sequence of  Bernoulli random variables of parameter $p$, independent of the martingale 
and define 
\begin{equation}
\phi = \sum_n X_n d_n . 
\end{equation}
We consider the function $\chi(p) = \mathbb{E} \phi^3 $, as a polynomial in $p$.   
It is exactly a quadratic polynomial in $p$, and we will show that it is approximately 
a cubic polynomial in $p$.  From this, we will be able to deduce the Theorem.

\begin{lemma} \label{l:3mart} We have $\chi(p) = M_1 p + M_2 p^2 $, where 
\begin{align} \label{e:M1} 
M_1 &=  \sum_{n}  \mathbb{E}    d_{n}^3, 
\\
M_2  & = 3\sum_{m<n} \mathbb{E} d_m d_n^2. 
\end{align}

\end{lemma}

\begin{proof} 
Expand $\phi $ into martingale differences and form the third power. It is a sum of terms 
\begin{equation}
\mathbb{E} \prod _{j=1}^3 X_{n_j} d_{n_j} 
\end{equation}
for integers $n_1, n_2, n_3$. Observe that if the maximal $n_j$ occurs just once, 
the expectation over $\Omega$ is zero.  Thus, the maximal $n_j$ must occur either 
exactly twice, or all three times.  If it occurs all three times, the 
expectation leads to the term $M_1p$.  If it occurs twice, we get a second degree term in $p$, 
with coefficient $M_2$ above.   
\end{proof}

We now use the hypothesis that \eqref{e:Ebig} fails.  To ease our considerations, 
we will write $A \simeq _\eta B$ if $\lvert A-B \rvert \leq C \eta^t \mathbb{P}(V)$ 
for some absolute choice of constants $C,t>0$. We will not keep track of them, and they can change from time to time.

\begin{lemma} Assume that \eqref{e:Ebig} fails. Then, 
we have 
\begin{equation}
 \chi (p) \simeq_\eta  (\mathbb{P}(V) +2M_1) p^3 -3M_1 p^2  + M_1 p.  
\end{equation} 
And, $ \chi (\tfrac12) \simeq_\eta   \tfrac18 \mathbb{P}(V)$.   
\end{lemma}

\begin{proof} Denoting expectation over the space for the Bernoulli random variables by $\mathbb{E}_X$, note that pointwise on $\Omega$, 
$
\mathbb{E}_X (\phi - p \mathbf{1}_V)^2 = p (1-p)(S \mathbf{1}_V)^2 
$. 
By the well known $L^p$ bound for the square function, we also have 
$\mathbb{E}\mathbf{1}_V (S \mathbf{1}_V)^3 \simeq_ \eta 0$.  
That means that 
$ \mathbb E \lvert \phi \mathbf{1}_V - p \mathbf{1}_V\rvert^3 \simeq_ \eta 0$.  
Indeed, we have
\begin{align} 
    \mathbb E \lvert \phi \mathbf{1}_V - p \mathbf{1}_V\rvert^3  
    & =
    \mathbb E \lvert \phi  \mathbf{1}_V  - p \mathbf{1}_V ^2  \rvert^3 
    \\  \label{e:KS}
    & = \mathbb E \Bigl\lvert  \mathbf{1}_V 
    \sum_n  (X_n -p) d_n \Bigr\rvert^3 
    \\ 
    & \lesssim 
    \mathbb E  \mathbf{1}_V  \Bigl[  
    \sum_n  d_n ^2  \Bigr]^{3/2} \simeq_ \eta 0. 
\end{align}
Above, we write $  \mathbf{1}_V =  \mathbf{1}_V ^2$, and expand one copy of $ \mathbf{1}_V$ in the martingale differences. Then appeal to Khitchine's inequality for balanced Bernoulli random variables to introduce the square function. 

 Continuing, write 
\begin{align}
\chi (p) = \mathbb{E} \phi^3 \mathbf{1}_V + \mathbb{E} (\phi-p\mathbf{1}_V)^3 \mathbf{1}_{\Omega \setminus V} . 
\end{align}
The first term is approximately $p^3 \mathbb{P}(V)$. 
For the second we note that 
\begin{align}
\mathbb{E} (\phi-p\mathbf{1}_V)^3 \mathbf{1}_{\Omega \setminus V}  
& \simeq_\eta \mathbb{E} (\phi-p\mathbf{1}_V)  ^3
\\
& = \sum_{n_1 , n_2, n_3} 
\mathbb{E} \prod _{j=1}^3 (X_{n_j} -p) d_{n_j} . 
\end{align}
In the expectation, if the maximal integer $n_j$ occurs only once, 
the expectation over $\Omega$ is zero. Thus, the maximal must occur 
twice, or three times.  

If the maximal $n_j$ occurs exactly twice,
a balanced Bernoulli random variable occurs once, and hence has expectation zero. 
The remaining  case requires $n_1 =n_2 = n_3$. That expression gives us 
\begin{align}
 \sum_{n} 
\mathbb{E}   (X_{n} -p) ^3 d_{n}^3 
&= p(1-p) (1-2p) M_1 
\\
& = (p-3p^2+2p^3) M_1.  
\end{align}
Recall that the third moment of a balanced Bernoulli random variable is $p(1-p)(1-2p)$. 
Above, the expression $M_1$ is defined in \eqref{e:M1}.   
And, note that the expression above vanishes if $p=1/2$.

\end{proof}

To conclude, note that we have the approximate polynomial identity 
\begin{equation}
(\mathbb{P}(V) +2M_1) p^3 -(M_2 +3M_1) p^2 \simeq _{\eta} 0.   
\end{equation}
Using this for two values of $\tfrac12< p < 1$, we see  that the coefficients 
above must be small.  
\begin{align}
\mathbb{P}(V) +2M_1 \simeq_\eta M_2 +3M_1 \simeq_\eta 0. 
\end{align}
Now, take the case of $p=\tfrac12$. We have 
\begin{align}
0 \simeq_{\eta} \tfrac18 \mathbb{P}(V) - \tfrac12 M_1 - \tfrac14 M_2 .
\end{align}
These last three approximate equations in $M_1$, $M_2$ and $\mathbb{P}(V) $ then imply $\mathbb{P}(V) \simeq_\eta 0 $. 
That is, we have $\mathbb{P}(V) = C \eta^t \mathbb{P}(V)$, for a 
bounded choice of constant $C$, and $0<t<1$.   
This shows that $\eta$ has an absolute lower bound. 

\begin{remark} Professor Jill Pipher \cite{P} points out that Lemma \ref{l:3mart} simplies if the martingale has symmetric differences. 
Namely, and using the notation of that Lemma, $M_1=0$ and $M_2 = 3 \mathbb{E} \mathbf{1}_V (S \mathbf{1}_V)^2$, so that one 
can just set $p=1$.   
\end{remark}

\section{Wavelet Square Function} 
\label{sec:wavelet_square_function}

The martingale inequality implies the same result for the dyadic square function, in any dimension.  
We believe a similar result holds for a wide gamut of square functions. 
The point of this section is to modify the argument for martingales to address a wavelet square function. 

We do not strive for the greatest generality. Let $w$ be a $C^2$ function, satisfying 
\begin{equation}
\label{e:rapid}  
\lvert w (x) \rvert \leq C_r (1+ \lvert x\rvert)^{-r}, \qquad r > 1,\ x\in \mathbb{R}. 
\end{equation}
Let $\mathcal D$ be a choice of dyadic intervals, and assume that 
$\{w_I \mid I\in \mathcal D\}  $ is an $L^2$ orthonormal basis, where 
\begin{equation}
w_I (x) = \lvert I\rvert^{-1/2} w \big((x-c_I)/\lvert I\rvert\big), 
\end{equation}
and $c_I$ is the center of $I$.  The corresponding wavelet square function is 
\begin{equation}
Sf ^2 = \sum_{I\in \mathcal D} \langle f, w_I\rangle ^2 w_I ^2 . 
\end{equation}

\begin{theorem}\label{t:wave} There is a constant $\eta>0$ so that for all sets $V\subset \mathbb{R}$ of finite measure, we have 
\begin{equation} \label{e:Wbig}
\int_V (S\mathbf{1}_V)^2 \;dx \geq \eta \lvert V\rvert. 
\end{equation}

\end{theorem}

We turn to the proof.  
We assume that the inequality above fails, and derive a lower bound 
for $\eta$.
 We work with a random wavelet projection of $\mathbf{1}_V$. 
Let $\{X_I \mid I\in \mathcal D\}$ be independent Bernoulli random variables 
with parameter $0<p<1$.  Set 
\begin{equation}
\phi = \sum_{I\in \mathcal D_n} X_I \langle \mathbf{1}_V, w_I \rangle w_I. 
\end{equation}
Then, define  $\chi (p) = \mathbb{E}\int \phi ^3 \;dx$. 

\begin{lemma} 
We have $\chi (p) =W_1 p +  W_2  p^2 + W_3  p^3$, where 

\begin{align} 
\label{e:W1} 
W_1 & = \mathbb{E}\int \sum_{I  \in \mathcal D_n} 
	\langle \mathbf{1}_V, w_{I} \rangle ^3 w_{I}^3 \;dx , 
\\
\label{e:W2}
W_2 & = 3 \, \mathbb{E} \int  \sum_{I_1 \neq  I_3 \in \mathcal D_n} 
	\langle \mathbf{1}_V, w_{I_1} \rangle ^2 \langle \mathbf{1}_V, w_{I_3} \rangle  w_{I_1} ^2 w_{I_3} \;dx , 
\\ 
\label{e:W3}
W_3 & = \mathbb{E} \int \sum_{\substack{I_1, I_2, I_3 \in \mathcal D_n \\ \textup{$I_j$ distinct}}} 
	\prod _{j=1}^3 \langle \mathbf{1}_V, w_{I_j} \rangle w_{I_j} \;dx . 
\end{align}
\end{lemma}

The term $W_3$ does not enter into the martingale case.

\begin{proof}
Write the third power as 
\begin{align}
\chi (p) & = \mathbb{E} \int \sum_{I_1, I_2, I_3 \in \mathcal D_n} 
\prod _{j=1}^3 
X_{I_j} \langle \mathbf{1}_V, w_{I_j} \rangle w_{I_j} \;dx . 
\end{align}
The sum above is divided into cases according to the number of 
distinct choices of the $I_j$.  First, if all the $I_j$ agree, 
\begin{align}
\mathbb{E} \int \sum_{I  \in \mathcal D_n} 
X_{I}^3 \langle \mathbf{1}_V, w_{I} \rangle ^3 w_{I}^3 \;dx 
=  p W_1. 
\end{align}
The probability that $X_I=1$ is $ p$. 
Second, if two of the $I_j$ agree, we have 
\begin{align}
3\mathbb{E} \int \sum_{I_1=I_2, I_3 \in \mathcal D_n} 
\prod _{j=1}^3 
X_{I_j} \langle \mathbf{1}_V, w_{I_j} \rangle w_{I_j} \;dx = W_2  p^2 
\end{align}
Third, we have the case where all the $I_j$ are distinct, which is the 
 $ p^3$ term.   

\end{proof}

We now use the hypothesis that \eqref{e:Wbig} fails. We use the same notation $\simeq_ \eta $ as before. 

\begin{lemma} Assume that \eqref{e:Wbig} fails. Then, 
\begin{equation}
 \chi (p) \simeq_\eta   p \lvert V\rvert + p (1-  p) (1-2 p)W_1,
\end{equation}
 $W_1$ is defined in \eqref{e:W1}.

\end{lemma} 

\begin{proof} 
The wavelet square function is bounded on $L^{12}$, say. See \cite{MR1162107}*{\S9.2}.  
Interpolating with the $L^2$ bound, we have
\begin{align}
\left[\int_V (S \mathbf{1}_V)^3 \;dx  \right]^{1/3}
& \leq \lVert \mathbf{1}_V S \mathbf{1}_V \rVert_2 ^{3/5}
\lVert S \mathbf{1}_V \rVert_{12} ^{2/5}
\\
& \leq \eta^{3/10} \lvert V\rvert ^{1/3}.  
\end{align}
We have $\mathbb{E} \phi = p \mathbf{1}_V$. 
Pointwise, we have the square function estimate 
\begin{equation}
\mathbb{E} \lvert \phi -  p\mathbf{1}_V \rvert^2 \lesssim (S \mathbf{1}_V)^2. 
\end{equation}
From Khintchine inequality, it follows that $\phi \mathbf{1}_V$ is very close to 
$ p\mathbf{1}_V$; 
indeed the inequalities of \eqref{e:KS} hold in this setting.
Thus, 
\begin{align}
\chi (p) &= \mathbb{E} \int \mathbf{1}_V \phi ^3 \;dx 
+ \mathbb{E} \int_{\mathbb{R} \setminus V}  ( \phi  -  p\mathbf{1}_V)^3  \;dx 
\\
& \simeq_\eta   p^3 \lvert V\rvert 
+ \mathbb{E} \int_{\mathbb{R} }  ( \phi  -  p\mathbf{1}_V)^3  \;dx. 
\end{align}
Expand the second term to 
\begin{equation}
\mathbb{E} \int \sum _{I_1, I_2, I_3 \in \mathcal D_n }\prod _{j=1}^3 
(X_{I_j} - p) \langle \mathbf{1}_V, w_{I_j} \rangle w_{I_j} \;dx . 
\end{equation}
If any $I_j$ occurs a single time, the expectation is zero by independence.
 Hence, the only contribution is when $I_1 = I_2 = I_3$. 
  Thus, recalling the third moment of a balanced Bernoulli,  
\begin{align}
\mathbb{E} \int_{\mathbb{R} }  ( \phi  -  p\mathbf{1}_V)^3  \;dx  
&=  p (1- p) (1-2 p) \int \sum_I 
\langle \mathbf{1}_V, w_I \rangle^3 w_I^3 \;dx  
\\
&= p (1-p) (1-2p) W_1. 
\end{align}
\end{proof}
 
An easy argument shows that $W_1$ is small.  

\begin{lemma} We have $W_1 \simeq_\eta 0$, 
where $W_1$ is defined in \eqref{e:W1}. 

\end{lemma}

\begin{proof} 
 Set 
\begin{equation}
\mathcal D' = \bigl\{ I\in \mathcal D \colon 
\lvert \langle \mathbf{1}_V, w_{I} \rangle\rvert 
\geq \eta^{1/3} \sqrt{ \lvert I\rvert}\bigr\}. 
\end{equation}
For $I\in \mathcal D'$, we have $\lvert I\cap V \rvert \gtrsim \eta^{1/3} 
\lvert I \rvert $.  Thus, 
\begin{align}
\sum_{I\in \mathcal D'} \langle \mathbf{1}_V, w_{I} \rangle^2 
&\lesssim \eta^{-2/3}  
\int_V (S \mathbf{1}_V) ^2 \;dx \leq \eta^{1/3} \lvert V\rvert . 
\end{align}
It is then easy to see that the Lemma holds.  Indeed, on the one hand, for intervals in $\mathcal D'$, we have 
\begin{equation}
\int \sum_{I\in \mathcal D'} 
\lvert \langle \mathbf{1}_V, w_I \rangle  \rvert^3 w_I^3 \; dx 
 \lesssim \int \sum_I 
\langle \mathbf{1}_V, w_I \rangle ^2 w_I^2 \; dx \lesssim \eta ^{1/3}  \lvert V \rvert.
\end{equation}
And, on the other hand, for those not in $\mathcal D'$, 
\begin{align}
\int \sum_{I\not\in \mathcal D'} 
\lvert \langle \mathbf{1}_V, w_I \rangle  \rvert ^{3}w_I^3 \; dx 
 \lesssim \eta ^{1/3} \int \sum_I 
\langle \mathbf{1}_V, w_I \rangle ^2 w_I^2 \; dx \lesssim \eta ^{1/3}  \lvert V \rvert.
\end{align}

\end{proof}

To conclude the proof of the Theorem, combine the three Lemmas above.  For $0< p <1$, we have 
\begin{equation}
W_2 p^2 +W_3 p^3   \simeq_\eta  p \lvert V \rvert. 
\end{equation}Above, we have a third degree polynomial in $p$, which is approximately zero for $0<p<1$. That forces all the coefficients to be small.
This implies that $ \lvert V \rvert \simeq _\eta  0$, that is $\lvert V \rvert \leq C \eta ^t \lvert V \rvert$, 
for absolute choices of $C>0$ and $0<t<1$.   
We conclude  that $ \eta $ admits an absolute lower bound.

\section{History and Open Problems} 
\label{sub:the_haar_shift_applied_to_indicator_sets}

\Para 
The Hilbert transform $H$, appropriately normalized, satisfies $H^2 = -I$. 
This with the parallelogram inequality then shows that $\lVert f \pm H f\rVert_2^2  =  2\lVert f \rVert_2 ^2$.  
So, as long as $f\neq 0$,  $Hf$ cannot be close to $f$ or $-f$.  
Stein and Weiss \cite{MR0107163} showed that the Hilbert transform, applied to 
an indicator set, has distribution that only depends upon 
the measure of the set.  
This observation has been reexamined by several authors 
\cites{MR209918,MR2559057,MR2900476,MR2721783,MR2310545}. 
Laeng \cite{MR2900476} shows that the distribution of $ H \mathbf{1}_V$ 
both inside and outside of $V$ depends only on the measure of $V$.  
The `noncomputational' argument of Calder\'on \cite{MR209918}*{pg 434} can be modified to prove this. 
Os\c{e}kowski \cite{MR3233978} uses a stochastic analysis approach to extend 
this result to   Riesz transforms and functions taking values in the interval $[0,1]$. 

\Para 
In a different direction, Tolsa and Verdera \cite{MR2248827} address closely related questions 
in which the Cauchy transform is applied to measures on the plane.  
This subject then concerns so called reflectionless measures \cite{MR3829611}.  

\Para We do not know of any closely related results about discrete operators.  
Several questions arise.  
Square functions come in many different forms. Surely many of them satisfy estimates like those in this paper.  
Similarly, one could consider the square functions of this paper with, say, weights in a Muckenhoupt class. 

\Para Besides the many variants of one parameter square functions, 
one can consider tensor products of martingales.  To phrase a concrete question, 
let $\mathcal R = \mathcal D \times \mathcal D$ be the collection of dyadic rectangles in the plane.  Associate to $R= R_1 \times R_2 \in \mathcal R$, we have a Haar function 
\begin{equation}
h_{R_1 \times R_2} (x_1, x_2) = \prod _{j=1}^2 h _{R_j} (x_j) . 
\end{equation}
The square function is then defined in the obvious way, 
\begin{equation}
(S f)^2  = \sum_R \frac{ \langle f, h_R \rangle ^2 } {\lvert R\rvert} \mathbf{1}_R.  
\end{equation}

\smallskip 
\textbf{Question:} Does Theorem \ref{t:mart} hold for this square function?  

\smallskip

We state it this way so that one can avoid potential pitfalls associated with the 
general (nonhomogeneous) setting of a martingale (although these concerns are not present in the one parameter setting).
The most naive variants of the proofs of this paper do not seem to imply this result.  

\Para Switching perspectives, consider the Haar shift operator  given by $T  h _{I _{\pm}} =  \pm  \cdot h _{I _{\mp}}$.  
See Figure \ref{f:shift}. 
 The key properties are that $ T ^{\ast} = -T$ and $ T ^2 = - I$.   This choice is made to closely mimic properties of the Hilbert transform.
It follows that the the eigenvalues of $T$ are $\pm i$. 
And, one can then see that for any function $f$, 
$ \lVert Tf - f \rVert_2 = \lVert T f + f\rVert_2 $.   
Hence,    $ \langle T \mathbf{1}_V , \mathbf{1}_V  \rangle =0$.  
But notice that for a dyadic interval $I$,  $(T \mathbf 1_I)  \mathbf 1_I \equiv 0$.

\begin{figure}
\begin{tikzpicture}
\draw[<->]  (-0.6,0)  -- (2.6,0) ;   \draw (0,-.1) -- (0,.1);    \draw (2,-.1) -- (2,.1);
\draw (0,-.5) -- (.5,-.5) -- (.5,.5) --  (1,.5);  \draw[very thick] (1,-.5) -- (1.5,-.5) -- (1.5,.5) -- (2,.5);
\draw  (.5,-.9) node  {$ I _{-}$};  \draw  (1.5,-.9)  node   {$ I _{+}$};
\draw (1.5, -1) node (r0) {};  \draw (.5, -1.1) node (l0) {};
\draw[<->]  (-3.6,-2.5)  -- (-0.4,-2.5) ;   \draw (-3,-2.6) -- (-3,-2.4);    \draw  (-1,-2.6) -- (-1,-2.4);
 \draw[very thick] (-3,-3) -- (-2.5,-3) -- (-2.5,-2) -- (-2,-2) node (r2) {};
 \draw [->]  (r0) -- (r2) ;   \draw (-2.5, -1.6) node {$ T h _{I _{+}}$};
\draw[<->]  (2.4,-2.5)  -- (5.6,-2.5) ;   \draw (3,-2.6) -- (3,-2.4);    \draw  (5,-2.6) -- (5,-2.4);
 \draw  (4,-2) node (l2) {} -- (4.5,-2) -- (4.5,-3) -- (5,-3) ;
 \draw [->]  (l0) -- (l2) ;   \draw (4.5, -1.6) node {$ T h _{I _{-}}$};
\end{tikzpicture}
\caption{The Haar shift operator $ T$, for a dyadic interval $I$. At the top, a dyadic interval $ I$, with the graph of $ h _{I _{-}}$ and
$ h _{I _{+}}$. The later is graphed in a thick line. The image of the two functions under $ T$ appears at the bottom.}
\label{f:shift}
\end{figure}
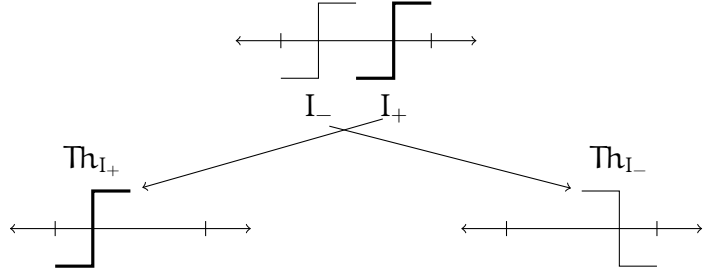

\smallskip 

\textbf{Question:} Is there an $\eta >0$ so that for all subsets $V\subset \mathbb R$ of finite measure, one has 
$ \int_V (T \mathbf 1_V)^2 \;dx < \eta \lvert V \rvert$?

\Para  Go to the two parameter situation, namely look at the tensor product 
$T \otimes T$, denoted $T_{[2]}$. Then, the eigenvalues 
are $\pm 1$, and this question appears to be far harder.  

\smallskip 

\textbf{Question:} Is there is a choice of $0< \eta <1$ so that for any subset $U\subset \mathbb R^2$ of finite measure, 
\begin{equation}
\int_U  \lvert T_{[2]} \mathbf1_U \rvert^2 \; dx \leq \eta  \lvert U \rvert  ? 
\end{equation}

\Para 
For a general orthonormal basis, one can phrase an associated square function. 
Notably in our approach, a higher moment than $2$ is used.  
Indeed, at a mininum the basis elements need to be `localized', 
and some super-orthogonality considerations appear to be necessary.

\end{document}